\documentclass[12pt]{article}
\usepackage{latexsym, amsmath, amsfonts, amsthm, amssymb}
\usepackage{bbm}
\usepackage{yhmath}
\usepackage{mathtools}
\usepackage{times}
\usepackage{a4wide}
\usepackage{times}
\usepackage{graphicx, tocloft}
\usepackage{mathrsfs}
\usepackage{tikz}
\usetikzlibrary{matrix,arrows}
\usepackage{ bbold }
\usepackage{comment}

\def \RR {\mathbb R}
\def \SS {\mathbb S}
\def \NN {\mathbb N}
\def \EE {\mathbb E}

\def \PP {\mathbb P}

\def \Tr {{ \rm Tr }}

\newcommand{\Var}{\mathrm{Var}}
\newcommand{\Cov}{\mathrm{Cov}}
\newcommand{\Vol}{\mathrm{Vol}}

\DeclarePairedDelimiter\abs{\lvert}{\rvert}%
\DeclarePairedDelimiter\norm{\lVert}{\rVert}%

\newtheorem{theorem}{Theorem}

\newtheorem{lemma}[theorem]{Lemma}

\newtheorem{corollary}[theorem]{Corollary}

\theoremstyle{definition}
\newtheorem{remark}[theorem]{Remark}

\newtheorem{theoremA}{Theorem} 

\def\qed{\hfill $\vcenter{\hrule height .3mm
		\hbox {\vrule width .3mm height 2.1mm \kern 2mm \vrule width .3mm
			height 2.1mm} \hrule height .3mm}$ \bigskip}

\begin{document}

\title{The slicing conjecture via small ball estimates}
\author{Pierre Bizeul\thanks{Technion – Israel Institute of Technology.\\
The author was supported by the European Research Council (ERC) under the European Union’s Horizon 2020 research and innovation programme, grant agreement No. 101001677 “ISOPERIMETRY”.}}
\date{\today}
\maketitle
\begin{abstract}
Bourgain’s slicing conjecture was recently resolved by Joseph Lehec and Bo’az Klartag. We present an alternative proof by establishing small ball probability estimates for isotropic log-concave measures. Our approach relies on the stochastic localization process and Guan’s bound, techniques also used by Klartag and Lehec. The link between small ball probabilities and the slicing conjecture was first observed by Dafnis and Paouris and is established through Milman's theory of M-ellipsoids.
\end{abstract}
\section{Introduction and result.}
Small ball probability estimates consist in finding suitable upper-bounds for the probability that a random vector $X$ in $\RR^n$ concentrates around a point $y\in\RR^n$. Specifically, we aim to bound the quantity
$$\PP(\norm{X-y}_2\leq r)$$
in a given range of $r>0$. The case when $X$ is an isotropic log-concave vector was pioneered by Paouris \cite{paouris2012small}. We say that a random vector $X$ in $\RR^n$ is $b$-subgaussian if for all $p\geq 1$ and all $\theta\in \SS^{n-1}$
\begin{equation*}
    \left(\EE\left|\left(X-\EE X\right)\cdot \theta\right|^p\right)^{1/p}\leq b\sqrt{p}.
\end{equation*}
The following result appears in \cite{paouris2012small} :
\begin{theoremA}[Paouris]\label{thm:paouris}
    Let $X$ be a $b$-subgaussian log-concave vector in $\RR^n$ with covariance matrix $A$. There exist a universal constant $c>0$ such that for any $\varepsilon<c$ and any $y\in\RR^n$,
    $$\PP\left(\norm{X-y}_2^2 \ \leq \ \varepsilon\Tr(A)\right)\ \leq \ \varepsilon^{\frac{c\Tr(A)}{b^2\norm{A}_{op}\norm{A^{-1}}_{op}}}.$$
\end{theoremA}
The theorem is tight, up to constants, when $X$ is the standard Gaussian for example. However, for general isotropic log-concave vectors, Paouris obtained a suboptimal exponent of order $\sqrt{n}$. We refer to \cite{paouris2012small} for details. 

This issue was partially resolved by Lee and Vampala \cite{lee2024eldan} using the stochastic localization process. We say that a random vector $X\in\RR^n$, or its law, is isotropic if $\EE[X]=0$ and $\EE[XX^T]=I_n$. A random vector $X$ is said to satisfy a Poincaré inequality if there exists a constant $K>0$ such that for all locally Lipschitz functions $f$,
$$\Var(f(X))\leq K\EE[\abs{\nabla f(X)}^2].$$
The smallest such constant $K>0$ is denoted by $C_P(X)$. We define 
$$\Psi_n^2 = \sup_{X}C_P(X)$$ the worst Poincaré constant of an isotropic log-concave vector in dimension $n$. Lee and Vampala \cite{lee2024eldan} proved :
\begin{theoremA}[Lee-Vampala]\label{thm:leevamp}
    Let $X$ be an isotropic log-concave vector in $\RR^n$. There exists a universal constant $c_B>0$ such that for any $y\in\RR^n$ and any $\varepsilon<c_B$,
    $$\PP\left(\norm{X-y}_2^2 \ \leq \ \varepsilon n\right)\ \leq \ \varepsilon^{\frac{c_Bn}{\Psi_n^2\log(n)}}$$
\end{theoremA}
The celebrated KLS conjecture, originally stated in \cite{kannan1995isoperimetric}, proposes that $\Psi_n$ is bounded above by a universal constant, independent of dimension. The best known bound on the KLS constant is due to Klartag \cite{klartag2023logarithmic}:
$$\Psi_n^2\leq C\log(n)$$
for some universal constant $C>0$. Plugging this into Theorem \ref{thm:leevamp} yields an exponent which is almost optimal, up to a factor $\log^2(n)$. If the KLS conjecture was resolved, Lee and Vampala estimate would still include a suboptimal factor $\log(n)$.

In this note, we remove the extra logarithmic factor and establish the following small ball estimate:
\begin{theorem}\label{thm:main}
    Let $X$ be an isotropic log-concave vector in $\RR^n$. There exists a universal constant $c_0>0$ such that for any $\varepsilon<c_0$ and any $y\in\RR^n$,
    $$\PP\left(\norm{X-y}_2^2 \ \leq \ \varepsilon n\right)\ \leq \ \varepsilon^{c_0n}.$$
\end{theorem}
As we now explain, Theorem \ref{thm:main} implies the slicing conjecture, which was very recently proved by Klartag and Lehec \cite{klartag2024affirmative}. We provide some background on the conjecture, and refer to \cite{klartag2022slicing} and \cite{klartag2024affirmative} for more details.

The slicing conjecture was originally formulated by Bourgain (\cite{bourgain1986high}, \cite{bourgain1986geometry}). It asks whether every convex body $K\in\RR^n$ of volume $1$ admits a hyperplane section with $n-1$ dimensional volume bigger than some universal constant, independent of dimension.

Following the terminology of \cite{klartag2024affirmative}, we say that a convex body $K\in\RR^n$ is in convex isotropic position if 
$$\Vol(K)=1, \quad\quad \int_K x \ dx=0, \quad\quad \int_K xx^T=L_K^2I_n.$$
for some constant $L_K$. Equivalently, the uniform measure on $L_K^{-1}K$ is isotropic. For $n\geq 1$, we define 
\begin{equation}
    L_n = \sup_{K} L_K
\end{equation}
where the supremum runs over all convex bodies $K\in\RR^n$ in convex isotropic position. Notice that the uniform measure on $L_K^{-1}K$ has density $L_K^n$ on its support. For a general isotropic log-concave density $f$ on $\RR^n$ we define 
\begin{equation}
    L_{f} = \sup_{x\in\RR^n}f(x)^{1/n},
\end{equation}
and for $n\geq1$
\begin{equation}
    \Tilde{L}_n=\sup_{f}L_f
\end{equation}
where the supremum runs over all isotropic log-concave densities $f$ on $\RR^n$. It is well known (see \cite{klartag2022slicing}) that the slicing conjecture may be reformulated using $L_n$ or $\Tilde{L}_n$.
\begin{lemma}\label{lem_198}
    Bourgain's slicing conjecture is equivalent to the following statements 
    \begin{itemize}
        \item $\sup_{n\geq1}L_n \ < \ +\infty$
        \item $\sup_{n\geq1}\Tilde{L}_n \ < \ +\infty$
    \end{itemize}
\end{lemma}
As already stated, after almost 40 years, the slicing conjecture was finally resolved by Klartag and Lehec \cite{klartag2024affirmative}.
\begin{theoremA}[Klartag-Lehec]\label{thm_slicing}
    $$\sup_{n\in\NN} L_n \ < \ +\infty.$$
\end{theoremA}
It was first observed by Danis and Paouris \cite{dafnis2010small} that Theorem \ref{thm:main} is actually equivalent to Theorem \ref{thm_slicing}. One direction is easy : if $L_n$ is is bounded above by some constant $C'>0$ then, by Lemma \ref{lem_198}, so is $\Tilde{L}_n$. Now, for any isotropic log-concave vector $X$ with density $f$, any $\varepsilon>0$ and $y\in\RR^n$,
\begin{align}
    \PP\left(\norm{X-y}_2^2 \ \leq \ \varepsilon^2 n\right)\ \leq \ (C')^n\Vol(\varepsilon\sqrt{ n}B_2^n) \ \leq \ \left(C''\varepsilon\right)^n 
\end{align}
which implies Theorem \ref{thm:main}. The other direction is more involved and relies on the theory of $M$-ellipsoid developped by Milman \cite{milman1985inverse}. We provide a proof in section \ref{subsec_slicing}. Note that this equivalence enables us to 'bootstrap' Theorem \ref{thm:main} into the a priori stronger statement:
\begin{corollary}
    Let $X$ be an isotropic log-concave vector in $\RR^n$. There exists a universal constant $C''>0$ such that for any $\varepsilon>0$ and any $y\in\RR^n$,
    $$\PP\left(\norm{X-y}_2^2 \ \leq \ \varepsilon^2 n\right)\ \leq \ \left(C''\varepsilon\right)^{n}.$$
\end{corollary}


Our approach for proving Theorem \ref{thm:main} is inspired by the work of Lee and Vampala \cite{lee2024eldan}. It uses Paouris's result and the stochastic localization process and relies on a very recent result by Guan \cite{guan2024note} about the behavior of the covariance matrix along this process. Finally, using $M$-ellipsoid and a result of Bourgain, Klartag and Milman about almost extremizers of the slicing constant $L_n$, we deduce Theorem \ref{thm_slicing} from Theorem \ref{thm:main}. 

While our proof of Theorem \ref{thm_slicing} shares key components with Klartag and Lehec's approach—such as the use of stochastic localization with Guan's bound and an $M$-ellipsoid argument—it differs in its reliance on small-ball estimates instead of stability estimates for the Shannon-Stam inequality.

\medskip

\noindent\textbf{Acknowledgements.} The author would like to thank Joseph Lehec, Bo'az Klartag and Galyna Livshyts for very interesting and useful discussions about this manuscript and the small ball problem.
\section{Preliminaries}
We begin with a simple consequence of Theorem \ref{thm:paouris}.
\begin{lemma}\label{lem_171}
    Let $X$ be a $b$-subgaussian log-concave vector with covariance matrix $A$, then for any $y\in\RR^n$ and any $\varepsilon<c$
    $$\PP\left(\norm{X-y}_2^2 \ \leq \ \varepsilon\Tr(A)\right)\ \leq \ \left(\frac{4n\norm{A}_{op}}{\Tr(A)}\varepsilon\right)^{\frac{c\Tr(A)^3}{8n^2b^2\norm{A}_{op}^2}}$$
    where $c>0$ is the constant from Theorem \ref{thm:paouris}.
\end{lemma}
\begin{proof}
    Let $\lambda_1\geq\dots\geq\lambda_n$ be the eigenvalues of the matrix $A$, and let $2\leq k\leq n$,
    \begin{equation}\label{eq_248}
        n\lambda_k \geq \sum_{i=k}^n \lambda_i = \Tr(A)-\sum_{i=1}^{k-1}\lambda_i \geq \Tr(A) - k\lambda_1.
    \end{equation}
    We choose $$k=\lfloor \frac{\Tr(A)}{2\lambda_1} \rfloor\vee1.$$
    If $k\geq2$ we get from \eqref{eq_248}
    $$\lambda_k \geq \frac{\Tr(A)}{2n},$$
    which remains true when $k=1$. We write $E$ for the $k$ dimensional subspace associated to the eigenvalues $\lambda_1,\dots,\lambda_k$, $X_E=P_EX$ for the projection of $X$ onto $E$, and $A_E=P_EAP_E$ for the covariance of $X_E$. Clearly, $X_E$ is $b$-subgaussian and its covariance satisfies :
    $$\norm{A_E}_{op}\leq \norm{A}_{op}; \quad\quad \norm{A_E^{-1}}_{op}\leq \frac{2n}{\Tr(A)}; \quad\quad \Tr(A_E)\geq\frac{k}{n}\Tr(A) \geq \frac{\Tr(A)^2}{4n\norm{A}_{op}}. $$
    Now, let $y\in\RR^n$ and write $y_E=P_Ey$. For $\varepsilon<c$
    \begin{align*}
        \PP\left(\norm{X-y}^2\leq \varepsilon\Tr(A)\right) &\leq \PP\left(\norm{X_E-y_E}^2\leq \varepsilon\Tr(A)\right) \\
        &\leq \left(\frac{\varepsilon\Tr(A)}{\Tr(A_E)}\right)^{\frac{c\Tr(A_E)}{\norm{A_E}_{op}\norm{A_E^{-1}}_{op}}}\\
        &\leq \left(\frac{4n\varepsilon\norm{A}_{op}}{\Tr(A)}\right)^{\frac{c\Tr(A)^3}{8n^2\norm{A}_{op}^2}}
    \end{align*}
    where for the second inequality we applied Paouris's result to $X_E$.
\end{proof}
We now recall two classical facts about log-concave vectors. The first one is known as Borell's Lemma
\begin{lemma}\label{lem_borell}
    There is a constant a constant $C_0>1$ such that for any log-concave real random variable $Y$ and any $p\geq2$
$$(\EE\abs{Y}^p)^{2/p} \leq C_0\EE\left(\abs{Y}^2\right).$$
\end{lemma} 
We say that a vector, or its density, is $t$-strongly log-concave for some $t>0$ if it is log-concave relatively to a Gaussian of variance $\frac{1}{t}$. The following lemma follows from a log-concave/convex correlation inequality of Hargé \cite{harge2004convex}.
\begin{lemma}\label{lem:barky_subgauss}
    Let $X$ be a $t$-strongly log-concave random vector. Then it is $\frac{1}{\sqrt{t}}$-subgaussian.
\end{lemma}
\begin{proof}
    Without loss of generality we assume that $t=1$ so that $X$ is log-concave with respect to the standard Gaussian measure denoted by $\gamma$. We may also assume that $\EE X=0$. Hargé's result state that if $f$ is log-concave and $g$ is convex,
    $$\Cov_{\gamma}(f,g)\leq 0.$$
    Choosing $f$ to be the relative density of $X$ with respect to $\gamma$ and $g(x) = \left(x\cdot\theta\right)^{p}$ for all $p\geq1$ and all $\theta\in\SS^{n-1}$ yields the result.
\end{proof}
\begin{remark}
    It would be enough for us to show that $X$ is $\frac{c''}{\sqrt{t}}$-subgaussian, for some universal constant $c''>0$, in which case many proofs are available. Assume without loss of generality that $t=1$. Notice that the inequality is one dimensional as, by Prékopa-Leindler, for any $\theta\in\SS^{n-1}$ the random variable $X\cdot\theta$ is still strongly log-concave. One can then for example use the one dimensional version of Cafarelli contraction theorem, which is easily proved. Alternatively, one can show a log-sobolev inequality for $X\cdot\theta$ and conclude using Herbst's argument.
\end{remark}
It is standard that, in proving Theorem \ref{thm:main}, one may assume that $X$ is supported in a ball of radius proportional to $\sqrt{n}$ and that $y=0$. We give a detailed proof of this reduction.
\subsection{Reduction to small diameter and to $y=0$}\label{subsec_reduction}
Let $X$ be an isotropic log-concave vector in $\RR^n$. Our first step is to consider the symmetrized version 
$$X_1 = \frac{X-\Tilde{X}}{\sqrt{2}}.$$
Where $\Tilde{X}$ is an independent copy of $X$. Thus $X_1$ is an isotropic symmetric log-concave vector, and for any $r>0$,

\begin{align}
    \PP\left(\norm{X-y}_2\leq r\right) &= \left[\PP\left(\norm{X-y}\leq r, \norm{\Tilde{X}-y}
    \leq r\right)\right]^{1/2}\nonumber\\
    &\leq \left[\PP\left(\norm{\frac{X-\Tilde{X}}{\sqrt{2}}}\leq \sqrt{2}r\right)\right]^{1/2}\nonumber\\
    &\leq \left[\PP\left(\norm{X_1}_2\leq \sqrt{2}r\right)\right]^{1/2}.\label{eq_274}
\end{align}
We set $K=B_2(0,2C_0\sqrt{n})$, where $C_0>1$ is the constant from Lemma \ref{lem_borell} and we define $$X_2 = X_1 \mid X_1 \in K.$$
If $f_1$ is the density of $X_1$, $X_2$ has density
\begin{equation}\label{eq_280}
    f_2 = \frac{f_1\mathbb{1}_K}{\int_K f_1} \leq 2f_1\mathbb{1}_K.
\end{equation}
Where we used that $X_1$ is isotropic, and Markov's inequality :
$$\int_K f_1 = \PP(\norm{X_1}^2\leq 4C_0^2n) \geq 1-\frac{1}{4C_0^2}\geq \frac{1}{2}.$$
$X_2$ is symmetric, hence, it is centered. From \eqref{eq_280} we deduce that 
\begin{equation*}
    \Cov(X_2)\leq 2I_n.
\end{equation*}
We claim that
\begin{equation}\label{eq_290}
    \frac{1}{2}I_n \ \leq \Cov(X_2) \ \leq 2I_n.
\end{equation}
To see the left-hand side, we write for any $\theta\in\SS^{n-1}$ :
\begin{align*}
    \int_{\RR^n} \langle x,\theta\rangle^2f_2(x)\ dx &\geq \int_{K} \langle x,\theta\rangle^2f_1(x)\ dx \\
    & = 1 - \int_{K^c} \langle x,\theta\rangle^2f_1(x)\ dx \\
    & \geq 1 - \PP(X_1\in K^c)^{1/2} \left(\int_{\RR^n}\langle x,\theta\rangle^4f_1(x)\ dx\right)^{1/2} \\
    &\geq 1 - \frac{1}{2} = \frac{1}{2}
\end{align*}
where we used Cauchy-Schwarz, Borell's Lemma and Markov's inequality. Finally, we set 
$$X_3 = \Cov(X_2)^{-1/2}X_2.$$
$X_3$ is isotropic by construction. Furthermore, using \eqref{eq_290}, we observe that has support inside $B(0,2\sqrt{2}C_0\sqrt{n})$. Finally, for any $0\leq\varepsilon<1$, setting $r=\varepsilon\sqrt{n}$
\begin{align}
    \PP\left(X \in B_2(y,r)\right)^2 & \leq \PP\left(X_1\in B_2(0,\sqrt{2}r\right)\nonumber \\
    &\leq \PP\left(X_2\in B_2(0,\sqrt{2}r\right)\nonumber  \\
    &=\PP\left(X_3\in \Cov(X_2)^{-1/2}B_2(0,\sqrt{2}r\right) \nonumber \\
    &\leq \PP\left(X_3\in B_2(0,2r)\right)\label{eq_finalreduction}
\end{align}
where we used successively \eqref{eq_274}, the fact that $B_2(0,\sqrt{2r})\subset K = B_2(0,2C_0\sqrt{n})$ and finally in the last line the inequality \eqref{eq_290}. Now it is clear from \eqref{eq_finalreduction} that it is enough to prove Theorem \ref{thm:main} when $X$ has support inside a ball of diameter $\sqrt{C_1n} :=\sqrt{8C_0^2n}.$

\medskip

The remainder of this section is dedicated to the stochastic localization process. It was introduced by Eldan in \cite{eldan2013thin} and used to tackle many problems in high-dimensional probability. We provide a brief introduction to it and refer to \cite{klartag2024isoperimetric} for a detailed history and description.
\subsection{The stochastic localization process}
Consider a measure $\mu$ with a log-concave density $f$ with respect to the Lebesgue measure. For $t\in\RR$ and $\theta\in\RR^n$ we define a density 
\begin{equation}\label{eq:density_theta_t}
    f_{t,\theta}(x) = \frac{1}{Z_{t,\theta}}e^{-t\frac{\abs{x}^2}{2} +\theta\cdot x}f(x), \quad \quad x\in\RR^n
\end{equation}
where $Z_{t,\theta}$ is a normalizing factor. Remark that $f_{t,\theta}$ is $t$-strongly log-concave, as $f$ is log-concave. We also define $$a_{t,\theta} = \int_{\RR^n} x f_{t,\theta}(x) \ dx$$ and $$A_{t,\theta} = \int_{\RR^n} (x-a_{t,\theta})^{\otimes2}f_{t,\theta}(x) \ dx,$$ the barycenter and covariance matrix of $f_{t,\theta}$ respectively.

The tilt process $(\theta_t)_{t\geq 0}$ is defined as the solution of the stochastic differential equation:
\begin{equation}\label{eq:tilt_process}
    d\theta_t = a_{t,\theta_t}\ dt \ +\ dB_t
\end{equation}
where $(B_t)_{t\geq0}$ is a standard Brownian motion. The process $f_{t,\theta_t}$ is the stochastic localization process starting at $f$. In the following we abbreviate $f_t=f_{t,\theta_t}$ and similarly for the barycenter and covariance, $a_t$ and $A_t$. We also define $\mu_t$ to be the (random) measure with density $f_t$. It can be shown that for any $x\in\RR^n$, $f_t$ is an Itô process satisfying the following SDE :
\begin{equation}\label{eq:density_dynamics}
    df_t(x) = \langle (x-a_t)\ , \ dB_t\rangle f_t(x).
\end{equation}
Thus, by integrating \eqref{eq:density_dynamics}, for any integrable function $\varphi$, the process $(\int_{\RR^n}\varphi d\mu_t)_{t\geq0}$ is a martingale. In particular, for any $t\geq 0$
\begin{equation}\label{eq:expectation}
    \EE_\mu(\varphi) = \EE\EE_{\mu_t}(\varphi).
\end{equation}
Thus at time $t\geq0$ we have decomposed the original log-concave measure $\mu$ into a mixture of measures $\mu_t$ which are $t$-strongly log-concave.

The log-concave Lichnerowicz inequality (see \cite{ledoux2001concentration} and \cite{klartag2023logarithmic} for an improved version) implies in particular that 
\begin{equation}\label{eq:cov_almost_sure}
    A_t \leq \frac{1}{t} \quad \quad \text{a.s}
\end{equation}
in the sense of symmetric matrices. In a breakthrough work, Guan \cite{guan2024note}, using notably the improved aforementionned improved Lichnerowicz inequality from \cite{klartag2023logarithmic}, was able to obtain a lower-bound on the trace of $A_t$ up to a time of order $1$:
\begin{lemma}\label{lem:Guan}
    Assume that the starting measure $\mu$ is isotropic. Then, there exists a universal constant $c_1>0$ such that 
    $$\EE \Tr(A_{c_1}) \geq c_1 n.$$
\end{lemma}
\medskip
To conclude this section, we analyze the evolution of the measure of sets along the process.
\subsection{Bounding the shrinkage of sets}
In this subsection we assume that $\mu$ has a bounded support with diameter $D$. Let $S\subset\RR^n$ be a set of positive measure and let $g_t = \mu_t(S)$. Using \eqref{eq:density_dynamics}, we compute
\begin{align*}
    dg_t = \left(\int_{S}(x-a_t)d\mu_t\right)\cdot dB_t\ .
\end{align*}
Thus
\begin{align}
    d[g]_t &= \left|\int_{S}(x-a_t)d\mu_t\right|^2 \ dt \nonumber \\
    &\leq D^2 g_t^2 \ dt. \label{eq_239}
\end{align}
Using Itô's formula, we compute
\begin{equation}
    d\EE\log(g_t^{-1}) = \frac{1}{2}\frac{d[g]_t}{g_t^2} \leq \frac{D^2\ dt}{2},
\end{equation}
which integrates to 
\begin{equation}\label{eq_246}
    \EE\log(g_t^{-1}) \leq \log\left(g_0^{-1}\right) + \frac{D^2t}{2}.
\end{equation}
We arrive at the following :
\begin{lemma}\label{lem_299}
    For any measurable set $S$ and any $\lambda>1$
    $$\mu(S) \leq e^{\frac{D^2t}{2}}\mu_t(S)^{\frac{1}{\lambda}}$$
    with probability at least $1-\frac{1}{\lambda}$, where $(\mu_t)_{t\geq0}$ is the stochastic localization process starting at $\mu$.
\end{lemma}
\begin{proof}
    Assume that $\mu(S)>0$ and let $\lambda>1$. From \eqref{eq_246}, using Markov's inequality, we get that with probability at least $1-\frac{1}{\lambda}$
    $$\log\frac{1}{g_t} \leq \lambda\left(\log\frac{1}{g_0}+\frac{D^2t}{2}\right).$$
    Equivalently,
    $$\log g_0 \leq \frac{1}{\lambda}\log{g_t} + \frac{D^2t}{2}$$
    which is what we wanted to prove.
\end{proof}
\section{Proof of the main results}
\subsection{Proof of Theorem \ref{thm:main}}
By section \ref{subsec_reduction} it is enough to consider an isotropic log-concave vector $X\sim \mu$ with support of diameter $\sqrt{C_1n}$. Let $(\mu_t)_{t\geq0}$ be the stochastic localization process starting at $\mu$. We write $X_t$ for the random vector having law $\mu_t$. Let $c_1$ be the constant appearing in Lemma \ref{lem:Guan}. We define the event 
\begin{equation}
    E_0 = \{\Tr(A_{c_1})\geq \frac{c_1n}{2} \}.
\end{equation}
\begin{lemma}
    The event $E_0$ happens with probability at least
    $$\PP(E_0) \geq \frac{c_1^2}{2}.$$
\end{lemma}
\begin{proof}
    Denote by $p_0 = \PP(E_0)$ the probability we want to lower bound. Remember that we have \eqref{eq:cov_almost_sure} :
    $$A_{c_1}\leq \frac{1}{c_1} \quad \text{a.s}$$
    Thus,
    \begin{align*}
        c_1n\leq \EE\Tr(A_{c_1}) &= \EE\Tr(A_{c_1}\mathbb{1_{E_0} }) + \EE\Tr(A_{c_1}\mathbb{1_{E_0^c} }) \\
        &\leq \frac{np_0}{c_1} + \frac{c_1n}{2}
    \end{align*}
    and the lemma follows.
\end{proof}
Let $\varepsilon<c$ and set $S_\varepsilon = B_2(0,\sqrt{\varepsilon n})$. We choose $\lambda = \frac{4}{c_1^2}$. By Lemma \ref{lem_299}, the event 
$$E_1 = \{\mu(S_\varepsilon) \leq e^{\frac{C_1n}{2}}\mu_t(S_\varepsilon)^{\frac{1}{\lambda}} \}$$
happens with probability at least $1-\frac{1}{\lambda} = 1-\frac{c_1^2}{4}$, which implies that
$$\PP(E_0\cap E_1) > 0.$$
By Lemma \ref{lem_171} and Lemma \ref{lem:barky_subgauss}, we get that on the event $E_0\cap E_1$,
\begin{align*}
    \mu(S_\varepsilon) &\leq e^{\frac{C_1n}{2}}\mu_t(S_\varepsilon)^{\frac{1}{\lambda}}\\
    &\leq e^{\frac{C_1n}{2}}\left(\frac{16\varepsilon}{c_1^3}\right)^{\frac{cc_1^8n}{256}} \\
    &\leq \varepsilon^{\frac{cc_1^8n}{512}}
\end{align*}
for $\varepsilon \leq \frac{c_1^6}{256e^{C_1}}\wedge c$. Thus Theorem \ref{thm:main} is proved with
$$c_0 = \min\left(c, \ \frac{c_1^6}{256e^{C_1}}, \ \frac{cc_1^8n}{512}\right) $$

\subsection{Proof of Theorem \ref{thm_slicing}}\label{subsec_slicing}
Let $K\in\RR^n$ be a convex body in convex isotropic position:
$$\Vol(K)=1 \quad \quad \Cov(K) = L_K^2I_d,$$
with $L_K=L_n$. The existence of such an extremizer is classical, and obtained via a compactness argument. It was shown in \cite{milman2004symmetrization} that $K$ is in $M$-position. Namely, there exists some constants $C_2,c_2>0$ such that 
\begin{equation}\label{eq:M_pos}
    \Vol\left(K\cap \sqrt{C_2n}B_2^n\right) > c_2^n.
\end{equation}
However by applying Theorem \ref{thm:main} to the uniform measure on $L_K^{-1}K$, which is log-concave and isotropic, we get that for any $0<c_3\leq c_0$,
\begin{equation}\label{eq_391}
    \Vol\left(K\cap L_K\sqrt{c_3n}B_2^n\right)\leq c_3^{c_0n}.
\end{equation}
Choosing $c_3 = \min\left(c_0,c_2^{\frac{1}{c_0}}\right)$, and comparing \eqref{eq:M_pos} and \eqref{eq_391}, we get 
\begin{equation}
    L_n^2 = L_K^2 < \frac{C_2}{c_3}.
\end{equation}

\bibliographystyle{plain}
\bibliography{biblio}
\end{document}